\RequirePackage{fix-cm}
\documentclass[smallextended]{svjour3}       
\usepackage{graphicx}
\usepackage{dsfont}

\usepackage{psfrag}

\usepackage{calrsfs}

\usepackage{mathrsfs}

\usepackage{pdfsync}
\usepackage{amsmath,amssymb, amsfonts}
\usepackage[shortlabels]{enumitem}

\usepackage{url}

\usepackage{float}

\usepackage[usenames]{color}

\usepackage{indentfirst,calc,euscript}

\usepackage{setspace}
\usepackage[reals]{layout}
\usepackage{xr}

\usepackage{sgame}

\usepackage{subfigure}
\usepackage{slashbox}
\usepackage{array}

\usepackage[colorlinks=true,breaklinks=true,bookmarks=true,urlcolor=blue,
     citecolor=blue,linkcolor=blue,bookmarksopen=false,draft=false]{hyperref}

\newcommand{\card}{\operatorname{Card}}

\newcommand{\ignore}[1]{}

\newcommand{\m}{\mathbb}

\DeclareMathOperator*{\cupp}{\cup}
\DeclareMathOperator*{\capp}{\cap}

\usepackage[authoryear]{natbib}
\begin{document}

\title{General limit value in zero-sum stochastic games
}


\author{Bruno Ziliotto    
}


\institute{Bruno Ziliotto \at
                CEREMADE, Universit\' e Paris Dauphine, Place du Mar\' echal de Lattre de Tassigny, 75016 Paris, France. \\
                         \email{ziliotto@math.cnrs.fr
}          
}


\maketitle
\bibliographystyle{abbrvnat}
\begin{abstract} 
\cite{BK76} and \cite{MN81} have respectively proved the existence of the asymptotic value and the uniform value in zero-sum stochastic games with finite state space and finite action sets. In their work, the total payoff in a stochastic game is defined either as a Cesaro mean or an Abel mean of the stage payoffs. The contribution of this paper is twofold: first, it generalizes the result of \cite{BK76} to a more general class of payoff evaluations, and it proves with an example that this new result is tight. It also investigates the particular case of absorbing games. Second, for the uniform approach of Mertens and Neyman, this paper provides an example of absorbing game to demonstrate that there is no natural way to generalize their result to a wider class of payoff evaluations. 
\end{abstract}
\section*{Notations}
\begin{itemize}
\item
The notation ``$X:=Y$" means ``$X$ is defined by the expression $Y$".
\item
If $A \subset B$, the complementary of $A$ in $B$ is denoted by $B \setminus{A}$.
\item
The set of non-negative integers is denoted by $\m{N}$, and $\m{N}^*:=\m{N} \setminus \left\{0\right\}$. 
\item
The set of real numbers is denoted by $\m{R}$, and the set of strictly positive real numbers is denoted by $\m{R}^*_+$.
\item
If $\left(C,\mathcal{C}\right)$ is a measurable space, we denote by $\Delta(C)$ the set of probability measures on $C$. We call $\delta_c$ the Dirac measure at $c \in C$. If 
$C_0 \subset C$ is a finite set and $(\alpha_c)_{c \in C_0} \in \Delta(C_0)$, then $\sum_{c \in C_0} \alpha_c \delta_c$  is denoted by $\sum_{c \in C_0} \alpha_c \cdot c$.
\end{itemize}
\section*{Introduction}
\label{intro}
Zero-sum stochastic games were introduced by \cite{SH53}. In this model, two players repeatedly play a zero-sum game, which depends on the state of nature. At each stage, a new state of nature is drawn from a distribution based on the actions of players and the state of the previous stage. The state of nature is announced to both players, along with the actions of the previous stage. Unless mentioned explicitly, we consider \textit{finite} stochastic games: the state space and the action sets are assumed to be finite.

There are several ways to evaluate the payoff in a stochastic game. For $n \in \m{N}^*$, the payoff in the $n-stage \ game$ is the Cesaro mean $\frac{1}{n} \sum_{m=1}^n g_m$, where $g_m$ is the payoff at stage $m \geq 1$. For $\lambda \in (0,1]$, the payoff in the $\lambda-discounted \ game$ is the Abel mean $\sum_{m \geq 1} \lambda(1-\lambda)^{m-1} g_m$.

Two main approaches are used to understand the properties of stochastic games with long duration :
\begin{itemize}
\item[-]
The asymptotic approach aims at determining if the value $v_n$ of the $n$-stage game and the value $v_{\lambda}$ of the $\lambda$-discounted game converge respectively when $n$ goes to infinity and $\lambda$ goes to 0. \cite{BK76} have proved that in finite stochastic games, $(v_n)$ and $(v_{\lambda})$ converge to the same limit. This result cannot be extended to stochastic games with compact action sets (see \cite{vigeral13}). Neither can it be extended to the case of public imperfect observation of the state of nature (see \cite{Z13}).
\item[-]
The uniform approach  analyzes the existence of strategies that are approximately optimal in any $n$-stage game and $\lambda$-discounted game, provided that $n$ is big enough and $\lambda$ is small enough. When this is the case, the stochastic game is said to have a \textit{uniform value}. \cite{MN81} have shown that finite stochastic games have a uniform value. Note that the existence of the uniform value implies the existence of the asymptotic value.
\end{itemize}
In this paper we investigate these two approaches, when payoffs are not restricted to be Cesaro means or Abel means of stage payoffs. As in \cite{CLS12}, if $\pi:=(\pi_m)_{m \geq 1} \in \Delta(\m{N}^*)$ is a sequence of weights, the payoff in the $\pi-weighted \ game$ is defined as the weighted sum $\sum_{m \geq 1} \pi_m g_m$. Intuitively, a $\pi$-weighted game with long duration corresponds to the case where the $(\pi_m)_{m \geq 1}$ are close to 0 (but still summing to one), but there are many different ways to define the convergence of $\pi$ to $0$. Once a criterion of convergence is defined, the asymptotic approach consists in determining whether or not the value $v_{\pi}$ of the $\pi$-weighted game converges when $\pi$ goes to 0. When this is the case, the game is said to have a \textit{general asymptotic value} (with respect to the chosen criterion).  Likewise, the uniform approach deals with the existence of strategies that are approximately optimal in any $\pi$-weighted game, with $\pi$ small enough. When this is the case, the game is said to have a \textit{general uniform value} (with respect to the chosen criterion).
Two main results can be found in literature:
\begin{itemize}
\item[-]
If $(\pi_m)_{m \geq 1} \in \Delta(\m{N}^*)$ is decreasing with respect to $m$, and if the criterion of convergence is: $\pi_1$ goes to 0, then finite stochastic games have a general uniform value (and thus a general asymptotic value). This result stems from the existence of the uniform value, established by \cite{MN81}, and from Theorem $1$ and Remark $(4)$ in \cite{NS10}.
\item[-]
 \cite{RV12} examine payoff weights that are not necessarily decreasing with respect to time, and consider the \textit{impatience} $I_1(\pi):=\sum_{m \geq 1} |\pi_{m+1}-\pi_m|$ of $\pi$ (see \cite[section 5.7]{S02b}).  They investigate the limit behavior of finite stochastic games with one Player (Markov Decision Processes) and finite POMDP (Markov Decision Processes with Partial Observation), when $I_1(\pi)$ goes to 0. In this framework, they show the existence of the general uniform value. Note that if $\pi$ is decreasing and $\pi_1$ goes to 0, then $I_1(\pi)=\pi_1$ goes to 0. Thus, for MDPs, this second result is more general than the first one.
\end{itemize}

In this paper, we also define a criterion on the convergence of $\pi$ to 0 under which the general asymptotic value exists in stochastic games. For the asymptotic approach, our theorem generalizes the two aforementioned results. In addition, we provide an example which shows first that our result is tight, and second that the result of \cite{RV12} cannot be extended to the Two-Player Case. We also show that for absorbing games with compact action sets and separately continuous transition and payoff functions, a sufficient condition under which $(v_{\pi})$ converges is that $\sup_{m \geq 1} \pi_m$ goes to 0 (when the action sets are finite, a sketch of proof for this last result is written in \cite{CLS12}). As for the uniform approach, we provide an example of absorbing game which shows that there is no natural way to relax the decreasing assumption on the weights.

The paper is organized as follows. Section 1 presents the model of stochastic games and some basic concepts. Section $2$ deals with the asymptotic approach, and Section $3$ presents the uniform approach.

\section{Generalities} \label{gen}
\subsection{Model of stochastic game} \label{model}
A stochastic game $\Gamma$ is defined by:
\begin{itemize}
\item[-]
A state space $K$,
\item[-]
An action set $I$ (resp. $J$) for Player 1 (resp. 2),
\item[-]
A payoff function $g:K \times I \times J \rightarrow [0,1]$,
\item[-]
A transition function $q:K \times I \times J \rightarrow \Delta(K)$.

\end{itemize}
Except in Subsection \ref{AG}, we assume that $K,I,J$ are (nonempty) finite sets.
\\
The initial state is $k_1 \in K$, and the stochastic game $\Gamma^{k_1}$ which starts in $k_1$ proceeds as follows. At each stage $m \geq 1$, both players choose simultaneously and independently an action, $i_m \in I$ (resp. $j_m \in J$) for Player 1 (resp. 2). The payoff at stage $m$ is $g_m:=g(k_m,i_m,j_m)$.
The state $k_{m+1}$ of stage $m+1$ is drawn from the probability distribution $q(k_m,i_m,j_m)$. Then $(k_{m+1},i_m,j_m)$ is publicly announced to both players.
\\

The set of all possible histories before stage $m$ is
$H_m:=(K \times I \times J)^{m-1} \times K$. A \textit{behavioral strategy} for Player 1 (resp. 2) is a mapping $\displaystyle \sigma:\cup_{m \geq 1} H_m \rightarrow \Delta(I)$ (resp. $\displaystyle \tau:\cup_{m \geq 1} H_m \rightarrow \Delta(J)$). The set of all behavioral strategies for Player 1 (resp. 2) is denoted by $\Sigma$ (resp. $\mathcal{T}$).
\\
A \textit{pure strategy} for Player 1 (resp. 2) is a mapping $\displaystyle \sigma:\cup_{m \geq 1} H_m \rightarrow I$ (resp. $\displaystyle \tau:\cup_{m \geq 1} H_m \rightarrow J$).
\\
A \textit{Markov strategy} is a strategy that depends only on the current stage and state. A Markov strategy for Player 1 (resp. 2) can be assimilated to a mapping from $\m{N}^* \times K$ to $\Delta(I)$ (resp. $\Delta(J)$).
\\
A \textit{stationary strategy} is a strategy that depends only on the current state. A stationary strategy for Player 1 (resp. 2) can be assimilated to a mapping from $K$ to $\Delta(I)$ (resp. $\Delta(J)$).
\\
The set of infinite plays of the game is $H_\infty:=(K \times I  \times J)^{\m{N}^*}$, and is equipped with the $\sigma$-algebra generated by cylinders.
A triple $(k_1,\sigma,\tau) \in K \times \Sigma \times \mathcal{T}$ induces a unique probability measure on $H_\infty$, denoted by $\mathbb{P}^{k_1}_{\sigma,\tau}$ (see \cite[Appendix D]{S02b}). Let $\pi \in \Delta(\m{N}^*)$ such that $\sum_{m \geq 1} \pi_m=1$. The $\pi-weighted \ game$ $\Gamma^{k_1}_\pi$ is the game defined by its normal form $(\Sigma,\mathcal{T},\gamma_{\pi}^{k_1})$, where 
\begin{equation*}
\gamma^{k_1}_{\pi}(\sigma,\tau):=\mathbb{E}^{k_1}_{\sigma,\tau}\left(\sum_{m \geq 1} \pi_m g_m \right).
\end{equation*}
By the minmax theorem (see \cite[Appendix A.5]{S02b}), the game $\Gamma^{k_1}_{\pi}$ has a value, denoted by $v_{\pi}(k_1)$:
\begin{equation*}
v_{\pi}(k_1)=\max_{\sigma \in \Sigma} \min_{\tau \in \mathcal{T}} \gamma^{k_1}_{\pi}(\sigma,\tau)
=\min_{\tau \in \mathcal{T}} \max_{\sigma \in \Sigma} \gamma^{k_1}_{\pi}(\sigma,\tau).
\end{equation*}
When for some $n \in \m{N}^*$, $\displaystyle \pi_m=n^{-1} 1_{m \leq n}$ for every $m \in \m{N}^*$, the game $\Gamma_n:=\Gamma_{\pi}$ is called the \textit{$n$-stage game}, and its payoff function is denoted by $\gamma_n$. When for some $\lambda \in (0,1]$, $\displaystyle \pi_m=\lambda(1-\lambda)^{m-1}$ for every $m \in \m{N}^*$, the game $\Gamma_{\lambda}:=\Gamma_{\pi}$ is called the \textit{$\lambda$-discounted game}, and its payoff function is denoted by $\gamma_{\lambda}$.

\subsection{Two results in the literature}
Let us fix a stochastic game $\Gamma$. Two standard definitions are recalled below:
\begin{definition}
The stochastic game $\Gamma$ has an \textit{asymptotic value} if the sequences $(v_{\lambda})$ and $(v_n)$ converge to the same limit, when respectively $\lambda$ goes to $0$ and $n$ goes to infinity.
\end{definition}
\begin{definition}
The stochastic game $\Gamma$ has a \textit{uniform value} $v_{\infty}:K \rightarrow [0,1]$ if for all $k_1 \in K$, for all $\epsilon>0$, there exists $(\sigma^*,\tau^*) \in \Sigma \times \mathcal{T}$ and $\bar{n} \in \m{N}^*$, such that for all $n \geq \bar{n}$ and $(\sigma,\tau) \in \Sigma \times \mathcal{T}$, we have
\begin{equation*}
\gamma^{k_1}_n(\sigma^*,\tau) \geq v_{\infty}(k_1)-\epsilon \quad \text{and} \quad
\gamma^{k_1}_n(\sigma,\tau^*) \leq v_{\infty}(k_1)+\epsilon.
\end{equation*}
\end{definition}
\cite{BK76} have proved that $\Gamma$ has an asymptotic value, and \cite{MN81} have generalized this result in the following way:
\begin{theorem} \label{MN}
The stochastic game $\Gamma$ has a \textit{uniform value} $v_{\infty}$.
In particular, $\Gamma$ has an \textit{asymptotic value}, and $(v_{\lambda})$ and $(v_n)$ converge to $v_{\infty}$, when respectively $\lambda$ goes to $0$ and $n$ goes to infinity. 
\end{theorem}
This theorem shows that both players have strategies that are approximately optimal in any long game $\Gamma^{k_1}_n$. The existence of a stronger notion of uniform value is then straightforward (see Theorem $1$ and Remark $(4)$ in \cite{NS10}): players have strategies that are approximately optimal in any game $\Gamma^{k_1}_{\pi}$ with $\pi=(\pi_m)_{m \geq 1}$ decreasing with respect to $m$ and $\pi_1$ sufficiently small. For completeness, we give a sketch of the proof of this corollary.
\begin{corollary} \label{MNd}
For all $k_1 \in K$, for all $\epsilon>0$, there exists $(\sigma^*,\tau^*) \in \Sigma \times \mathcal{T}$ and $\alpha>0$ such that for all $\pi=(\pi_m)_{m \geq 1} \in \Delta(\m{N}^*)$ decreasing with respect to $m$, and that satisfies $\displaystyle I_{\infty}(\pi):=\sup_{m \geq 1} \pi_m=\pi_1 \leq \alpha$, we have for all $(\sigma,\tau) \in \Sigma \times \mathcal{T}$
\begin{equation*}
\gamma^{k_1}_{\pi}(\sigma^*,\tau) \geq v_{\infty}(k_1)-\epsilon \quad \text{and} \quad
\gamma^{k_1}_{\pi}(\sigma,\tau^*) \leq v_{\infty}(k_1)+\epsilon.
\end{equation*}
In particular, 
$(v_{\pi})$ converges to $v_{\infty}$ when $\pi$ is decreasing and $I_{\infty}(\pi)$ goes to $0$. 
\end{corollary}
\begin{proof}{(Sketch)}
If $\pi \in \Delta(\m{N}^*)$ is decreasing, then the $\pi$-weighted payoff is a convex combination of Cesaro-mean payoffs:
\begin{equation*}
\sum_{m \geq 1} \pi_m g_m=\sum_{m \geq 1} m(\pi_m-\pi_{m+1}) \frac{1}{m} \sum_{l=1}^m g_l.
\end{equation*}
The proof of the corollary follows from this equality.
$\squareforqed$
\end{proof}
 In the One-Player Case, by a particular case of Theorem 3.19 in \cite{RV12}, this result can be extended to a wider class of weights, in the following way:
\begin{theorem} \label{RV}
Assume that $\Gamma$ is a Markov decision process, that is, the functions $q$ and $g$ do not depend on the action of Player $2$.
Then 
$\Gamma$ has a \textit{general uniform value}: 
for all $k_1 \in K$, for all $\epsilon>0$,  there exists $\sigma^* \in \Sigma$ and $\alpha>0$, such that for all $\pi \in \Delta(\m{N}^*)$ that satisfies
\\
$\displaystyle I_1(\pi):=\sum_{m \geq 1}|\pi_{m+1}-\pi_m| \leq \alpha$, we have
\begin{equation*}
v_{\infty}(k_1)-\epsilon \leq \gamma^{k_1}_{\pi}(\sigma^*) \leq v_{\infty}(k_1)+\epsilon .
\end{equation*}
In particular, 
$(v_{\pi})$ converges to $v_{\infty}$ when $I_1(\pi)$ goes to $0$.
\end{theorem}
In the next section, we study the asymptotic approach and investigate whether we can also relax the decreasing assumption in Corollary \ref{MNd} in the Two-Player Case. 
\section{Asymptotic approach}
\subsection{A criterion for the convergence of $(v_{\pi})$}
The first obvious point is that if one removes the decreasing assumption in Corollary \ref{MNd} and only assumes that $I_{\infty}(\pi):=\sup_{m \geq 1} \pi_m$ goes to zero, $(v_{\pi})$ does not necessarily converge. Indeed consider the Markov chain which oscillates deterministically between two states, one with payoff $1$, the other one with payoff $0$. Consider two sequences of weights, one which puts weight on even stages and one which puts weight on odd stages. The difference between these two payoff evaluations is always equal to $1$. 
Thus the condition $I_{\infty}(\pi) \rightarrow 0$ is not a sufficient condition to obtain the convergence of $(v_{\pi})$. Let us now provide a more restrictive criterion under which $(v_{\pi})$ converges.
\begin{definition}
Let $\pi \in \Delta(\m{N}^*)$ and $p \in (0,+\infty]$. The $\textit{p-impatience}$ of $\pi$ is the quantity $I_p(\pi) \in (0,+\infty]$ defined by

\begin{equation*}
I_p(\pi):= \left\{
\begin{array}{ll}
\displaystyle   \sum_{m \geq 1} \left|{(\pi_{m+1})^p}-{(\pi_m)^p} \right| & \mbox{if} \ \ p<\infty, \\
\displaystyle \sup_{m \geq 1} \pi_m & \mbox{if} \ \ p=\infty.
\end{array}
\right.
\end{equation*}
\end{definition}
When $I_p(\pi)$ is small, it means that players are very patient. When in addition $p<\infty$, it means that the variations of $\pi$ with respect to $m$ are small.

\begin{proposition} \label{imp}
Let $\pi \in \Delta(\m{N}^*)$ and $p,p' \in \m{R}^*_{+}$, such that $p \leq p'$. Then 
\begin{itemize}
\item[-]
$I_{p'}(\pi) \leq (p'/p) I_p(\pi)$,
\item[-]
$I_{\infty}(\pi) \leq  (I_p(\pi))^{1/p}$.
\end{itemize}
\end{proposition}
\begin{proof}
Let $m \in \m{N}^*$ and $q:=p'/p$. The Mean Value Theorem implies that
\begin{equation*}
\left|(\pi_{m+1})^{p'}-(\pi_m)^{p'} \right|=\left|\left[(\pi_{m+1})^{p}\right]^q-\left[(\pi_m)^{p}\right]^q \right| \leq q  \left|(\pi_{m+1})^{p}-(\pi_m)^{p} \right|,
\end{equation*}
and it yields: $I_{p'}(\pi) \leq q I_p(\pi)$. As for the second inequality, we have
\begin{equation*}
(\pi_m)^p=\sum_{m' \geq m} \left[(\pi_{m'})^p-(\pi_{m'+1})^p \right] \leq I_p(\pi),
\end{equation*}
and it yields: $I_{\infty}(\pi) \leq (I_p(\pi))^{1/p}$.
$\squareforqed$
\end{proof}
\begin{remark}
When $(\pi_m)_{m \geq 1}$ is decreasing, for all $p \in \m{R}_+^*$, we have $I_p(\pi)=(\pi_1)^p$. Consequently, given $p,p' \in \m{R}_+^*$ such that $p \leq p'$, there does not exist a real number $C(p,p')>0$ such that for all $\pi \in \Delta(\m{N}^*)$, $I_{p'}(\pi) \geq C(p,p') I_p(\pi)$. 
\end{remark}
Let us fix a stochastic game $\Gamma$, and let $v_{\infty}=\lim_{\lambda \rightarrow 0} v_{\lambda}=\lim_{n \rightarrow+\infty} v_n$ be its uniform value. For $f$ a real-valued function, denote by $\left\|f \right\|_{\infty}$ the supremum of $f$.
\begin{definition}
Let $p \in (0,+\infty]$. 
The stochastic game $\Gamma$ has a \textit{$p$-asymptotic value} if for all $\epsilon>0$, there exists $\alpha>0$ such that for all $\pi \in \Delta(\m{N}^*)$ verifying $I_p(\pi) \leq \alpha$, we have $\left\|v_{\pi}-v_{\infty}\right\|_{\infty} \leq \epsilon$. 
\end{definition}
\begin{remark} \mbox{}
\begin{itemize}
\item[-]
If for some $p' \in (0,+\infty]$, the game $\Gamma$ has a $p'$-asymptotic value, it has a $p$-asymptotic value for all $p \leq p'$. It results directly from Proposition \ref{imp}.
\item[-]
By Theorem \ref{RV}, any Markov decision process has a $1$-asymptotic value. 
\item[-]
Finite absorbing games have an $\infty$-asymptotic value (see \cite{CLS12}).
\item[-]
The Markov chain described at the beginning of this subsection has no $p$-asymptotic value for all $p>1$.
\end{itemize}
\end{remark}

Recall that $(v_{\lambda})$ can be expanded in Puiseux series (see \cite{BK76}): there exists $\beta>0$, $M \in \m{N}^*$ and $r_m \in \m{R}^{K}$ such that for all $k \in K$ and $\lambda \in [0,\beta)$
\begin{equation} \label{puiseux}
v_{\lambda}(k)=\sum_{m \geq 0} r_m(k) \lambda^{\frac{m}{M}} ,
\end{equation}
with the convention $v_0:=v_{\infty}$. 
\begin{definition}
Let $m_0=\inf\left\{m \geq 1 \ | \ r_m \neq 0 \right\}$. 
The quantity $s:=m_0/M \in[0,+\infty]$ is called the \textit{order} of $\Gamma$. 
\end{definition}
Note that if $s<+\infty$, there exists $C>0$ such that for all $(\lambda,\lambda') \in [0,\beta)^2$, we have
\begin{equation} \label{plip}
\left\|v_{\lambda}-v_{\lambda'}\right\|_\infty \leq C \left|\lambda^{s}-\lambda'^{s} \right| .
\end{equation} 

If $A$ is a finite set, the cardinal of $A$ is denoted by $\card A$. 
By Remark 3 in \cite{B14}, we have
\begin{equation*}
s \geq (\card K \card I)^{-\sqrt{\card K \card I}}.
\end{equation*}
Now we can state our main theorem. 
\begin{theorem} \label{ex}
The stochastic game $\Gamma$ has a $s$-asymptotic value. In particular, if $p \in \m{R}_{+}^*$ is smaller or equal to $(\card K \card I)^{-\sqrt{\card K \card I}}$, then $\Gamma$ has a $p$-asymptotic value.
\end{theorem}
\begin{proof}
Neyman has shown that in a stochastic game the convergence of $(v_n)$ can be deduced from the Shapley equation and the fact that $(v_{\lambda})$ is absolutely continuous with respect to $\lambda$ (see \cite[Theorem C.8, p. 177]{S02b}). We use similar tools. 

Let $\pi \in \Delta(\m{N}^*)$ and $r \in \m{N}$ such that there exists $m \geq r+2$, $\pi_m \neq 0$. A sequence of weights $\pi^r \in \Delta(\m{N}^*)$ is defined in the following way: for $m \in \m{N}^*$,

\begin{equation*}
\pi^r_m:= \left\{
\begin{array}{ll}
\displaystyle \frac{\pi_{m+r}}{\displaystyle \sum_{m' \geq r+1} \pi_{m'}}  & \mbox{if} \ \ \pi_{m+r} \neq 0, \\
0 & \mbox{if} \ \ \pi_{m+r}=0.
\end{array}
\right.
\end{equation*}

Let $\lambda_r:=\pi^r_1$. Let $k \in K$. Shapley equations yield (see \cite{CLS12}):
\begin{eqnarray} \label{dyn11}
v_{\pi^r}(k)&=& \max_{x \in \Delta(I)} \min_{y \in \Delta(J)}
\left\{ \lambda_r g(k,x,y)+(1-\lambda_r)\m{E}^k_{x,y}(v_{\pi^{r+1}}) \right\}
\\
\label{dyn12}
&=& \min_{y \in \Delta(J)} \max_{x \in \Delta(I)} 
\left\{ \lambda_r g(k,x,y)+(1-\lambda_r)\m{E}^k_{x,y}(v_{\pi^{r+1}}) \right\}
\end{eqnarray}
and
\begin{eqnarray} \label{dyn21}
v_{\lambda_r}(k)&=& \max_{x \in \Delta(I)} \min_{y \in \Delta(J)}
\left\{ \lambda_r g(k,x,y)+(1-\lambda_r)\m{E}^k_{x,y}(v_{\lambda_{r}}) \right\}
\\ 
\label{dyn22}
&=&  \min_{y \in \Delta(J)} \max_{x \in \Delta(I)}
\left\{ \lambda_r g(k,x,y)+(1-\lambda_r)\m{E}^k_{x,y}(v_{\lambda_{r}}) \right\},
\end{eqnarray}
where 
\begin{equation} \label{deftrans}
\m{E}^k_{x,y}(f):=\sum_{(k',i,j) \in K \times I \times J} x(i) y(j)q(k,i,j)(k') f(k')
\end{equation}
and 
\begin{equation} \label{defpayoff}
\displaystyle g(k,x,y):=\sum_{(i,j) \in I \times J} x(i) y(j) g(k,i,j).
\end{equation}
Note that these equations also hold when $\lambda_r=0$, with the convention $v_0:=v_{\infty}$. 
Let $x \in \Delta(I)$ be optimal in (\ref{dyn11}) and $y \in \Delta(J)$ be optimal in (\ref{dyn22}). We have
\begin{equation} \label{indyn1}
v_{\pi^r}(k) \leq \lambda_r g(k,x,y)+(1-\lambda_r)\m{E}^k_{x,y}(v_{\pi^{r+1}})
\end{equation}
and
\begin{equation} \label{indyn2}
v_{\lambda_r}(k) \geq \lambda_r g(k,x,y)+(1-\lambda_r)\m{E}^k_{x,y}(v_{\lambda_{r}}).
\end{equation}
Combining the two inequalities yields:
\begin{equation*}
v_{\pi^r}(k)-v_{\lambda_r}(k) \leq (1-\lambda_r)\left\|v_{\pi^{r+1}}-v_{\lambda_r}\right\|_\infty.
\end{equation*}
Symmetrically, if $x' \in \Delta(I)$ is optimal in (\ref{dyn12}) and $y' \in \Delta(J)$ is optimal in (\ref{dyn21}), then
\begin{equation*}
v_{\lambda_r}(k)-v_{\pi^r}(k) \leq (1-\lambda_r)\left\|v_{\pi^{r+1}}-v_{\lambda_r}\right\|_\infty,
\end{equation*} 
and thus
\begin{equation} \label{maj}
\left\|v_{\pi^r}-v_{\lambda_r}\right\|_\infty \leq (1-\lambda_r)\left\|v_{\pi^{r+1}}-v_{\lambda_r}\right\|_\infty.
\end{equation}
Let $\displaystyle \Pi_r:=\prod_{r'=0}^{r-1} (1-\lambda_{r'})=
\sum_{m \geq r+1} \pi_m$. Note that $\displaystyle \lim_{r \rightarrow+\infty} \Pi_r=0$. The last inequality yields:
\begin{equation*}
\Pi_r \left\|v_{\pi^r}-v_{\lambda_r}\right\|_\infty \leq
\Pi_{r+1} \left\|v_{\pi^{r+1}}-v_{\lambda_{r+1}}\right\|_\infty
+\Pi_{r+1} \left\|v_{\lambda_{r+1}}-v_{\lambda_r}\right\|_\infty.
\end{equation*}
Let $N \in \m{N}^*$ such that there exists $m \geq N+1$, $\pi_m \neq 0$. Summing this inequality over $r \in \left\{0,1,...,N-1 \right\}$ yields:
\begin{equation} \label{in1}
\left\|v_{\pi}-v_{\lambda_0}\right\|_\infty \leq 
\Pi_{N} \left\|v_{\pi^{N}}-v_{\lambda_{N}}\right\|_\infty
+ \sum_{r=1}^N \Pi_{r} \left\|v_{\lambda_{r}}-v_{\lambda_{r-1}}\right\|_\infty.
\end{equation}
Let $\epsilon \in (0,1)$. Let $N_0:=\max \left\{N \geq 1 \ | \ \Pi_{N} \geq \epsilon \right\}$. We have $\Pi_{N_0} \leq \epsilon+I_{\infty}(\pi)$. For $N=N_0$, inequality (\ref{in1}) writes:
\begin{equation} \label{in2}
\left\|v_{\pi}-v_{\lambda_0}\right\|_{\infty} \leq \epsilon+I_{\infty}(\pi) +\sum_{r=1}^{N_0} \left\|v_{\lambda_{r}}-v_{\lambda_{r-1}}\right\|_\infty.
\end{equation}
Assume that $I_{\infty}(\pi) < \epsilon \beta$ (see equation (\ref{puiseux}) for the definition of $\beta$). 
\\
Thus $\lambda_r \leq I_{\infty}(\pi)/\Pi_{N_0}<\beta$ for all $r \in \left\{0,1,...,N_0 \right\}$. If $s=\infty$, then
\\
 $\sum_{r=1}^{N_0} \left\|v_{\lambda_{r}}-v_{\lambda_{r-1}}\right\|_\infty=0$, and the last inequality proves that $\Gamma$ has an $\infty$-asymptotic value. Assume now that $s<\infty$. 
Using (\ref{plip}), let us majorize the term on the right in inequality (\ref{in2}):
\begin{equation*}
\sum_{r=1}^{N_0} \left\|v_{\lambda_{r}}-v_{\lambda_{r-1}}\right\|_\infty
\leq C \sum_{r=1}^{N_0} \left|(\lambda_{r})^s-(\lambda_{r-1})^s \right|.
\end{equation*}
Let $r \in \left\{1,2,...,N_0\right\}$. The quantity $\left|(\lambda_r)^s-(\lambda_{r-1})^s\right|$ is smaller than 
\begin{equation*}
\left|\frac{(\pi_{r+1})^s}{\left( \displaystyle \sum_{m \geq r+1} \pi_m \right)^s}-\frac{(\pi_r)^s}{\left(\displaystyle \sum_{m \geq r+1} \pi_m\right)^s}\right|+
\left|\frac{(\pi_{r})^s}{\displaystyle \left(\sum_{m \geq r+1} \pi_m \right)^s}-\frac{(\pi_r)^s}{\displaystyle \left(\sum_{m \geq r} \pi_m\right)^s}\right|.
\end{equation*}
By definition of $N_0$, we have 
\begin{equation*}
\sum_{m \geq r} \pi_m \geq \sum_{m \geq r+1} \pi_m \geq \epsilon.
\end{equation*}
Therefore we can majorize the term on the left by
$\epsilon^{-s} \displaystyle \left|(\pi_{r+1})^s-(\pi_r)^s\right|$. As for the term on the right, by the Mean Value theorem we have
\begin{eqnarray*}
\left(\displaystyle \sum_{m \geq r+1} \pi_m \right)^{-s}-\left(\sum_{m \geq r} \pi_m \right)^{-s} &\leq& 
s \left(\sum_{m \geq r+1} \pi_m \right)^{-1-s} \pi_r
\\
&\leq& s \epsilon^{-1-s} \pi_r.
\end{eqnarray*}
Finally we have
\begin{eqnarray*}
\sum_{r \geq 1} \left|(\lambda_r)^s-(\lambda_{r-1})^s\right| &\leq& \sum_{r \geq 1} \left(\epsilon^{-s}\left|(\pi_{r+1})^s-(\pi_r)^s\right|+s \epsilon^{-1-s} (\pi_r)^{1+s} \right) 
\\
&\leq& \epsilon^{-s} I_s(\pi)+s \epsilon^{-1-s} I_{\infty}(\pi)^s
\\
&\leq& 
(\epsilon^{-s}+s \epsilon^{-1-s})I_s(\pi).
\end{eqnarray*}
Plugging this into (\ref{in2}) gives
\begin{equation*}
\left\|v_{\pi}-v_{\lambda_0}\right\|_{\infty} \leq \epsilon+I_s(\pi)^{1/s} +C\left(\epsilon^{-s}+s \epsilon^{-1-s}\right)I_s(\pi).
\end{equation*} 
Thus for $I_s(\pi)$ sufficiently small, we have both
$\displaystyle \left\|v_{\pi}-v_{\lambda_0}\right\|_\infty \leq \epsilon$ and
\\
$\displaystyle \left\|v_{\lambda_0}-v_{\infty}\right\|_\infty \leq \epsilon$, which concludes the proof.
$\squareforqed$
\end{proof}
\begin{corollary} \label{ascv}
Let $(\pi^n) \in (\Delta(\m{N}^*))^{\m{N}}$ such that for all $p>0$, $\displaystyle \lim_{n \rightarrow +\infty} I_p(\pi^n)=0$. Then in any stochastic game, $(v_{\pi^n})_{n \geq 0}$ converges to $v_{\infty}$.
\end{corollary}
\begin{proof}
Let $\Gamma$ be a stochastic game of order $s \in (0,+\infty]$. By Theorem \ref{ex}, $\Gamma$ has a $s$-asymptotic value. By assumption, we have $\displaystyle \lim_{n \rightarrow +\infty} I_s(\pi^n)=0$, thus $(v_{\pi^n})_{n \geq 0}$ converges to $v_{\infty}$.
$\squareforqed$
\end{proof}
The following remarks show that for the asymptotic approach, Corollary \ref{ascv} is more general than Corollary \ref{MNd} and Theorem \ref{RV}.
\begin{remark} \label{rq} \mbox{} \\
\begin{itemize}
\item[-]
When $(\pi_m)_{m \geq 1}$ is decreasing, for all $p>0$, $I_p(\pi)=(\pi_1)^p$. According to Corollary \ref{ascv}, $(v_{\pi})$ converges when $\pi_1$ goes to 0 (compare with the asymptotic approach in Corollary \ref{MNd}).
\item[-]
When $s=1$, the mapping $\lambda \rightarrow v_\lambda$ is Lipschitz. For instance, this is the case when $\Gamma$ is a Markov decision process: see \cite[Chapter 5, Proposition 5.20]{S02b}. By Theorem \ref{ex}, $\Gamma$ has a $1$-asymptotic value (compare with the asymptotic approach in Theorem \ref{RV}). 
\item[-]
For $(l,n) \in \m{N}\times \m{N}^*$, let $\displaystyle \pi^{l,n}=n^{-1} 1_{l+1 \leq m \leq l+n}$. The $(\pi_{l,n})$ are non-monotonic sequences, thus Corollary \ref{MNd} does not apply. Nevertheless, $\displaystyle I_s(\pi^{l,n})=2 n^{-s}$. Consequently, for any $\epsilon>0$, there exists $\bar{n} \in \m{N}^*$ such that for all $n \geq \bar{n}$, for all $l \in \m{N}$, $I_s(\pi^{l,n}) \leq \epsilon$. By Theorem \ref{ex}, $\Gamma$ has a $s$-asymptotic value, and we deduce that
\begin{equation*}
 \lim_{n \rightarrow +\infty} \sup_{l \in \m{N}} \left\|v_{\pi^{l,n}}-v_{\infty} \right\|_{\infty}=0.
\end{equation*}
\end{itemize}
\end{remark}

\subsection{Absorbing games} \label{AG}
In this subsection, we relax the finiteness assumption on the action sets.
An absorbing state is a state such that once it is reached, the game remains in this state forever, and the payoff does not depend on the actions (\textit{absorbing payoff}). An absorbing game is a stochastic game that has at most one nonabsorbing state.
\\
\cite{MNR09} have proved the existence of the uniform value in absorbing games with compact action sets and separately continuous transition and payoff functions. In particular, $(v_{\lambda})$ converges.
Adapting the proof of the previous subsection, we prove the following proposition:

\begin{proposition} \label{exabs}
Let $\Gamma$ be an absorbing game with compact action sets and separately continuous transition and payoff functions. Then, $\Gamma$ has an $\infty$-asymptotic value.
\end{proposition}
\begin{remark}
For finite $I$ and $J$, this result was stated in \cite{CLS12}, with a sketch of proof. Here, we provide a complete and simpler demonstration, which holds in a more general framework. 
\end{remark}
\begin{proof}
Again, we adopt the convention $v_0:=v_{\infty}$.
The Shapley equations (\ref{dyn11}), (\ref{dyn12}), (\ref{dyn21}) and (\ref{dyn22}) still hold true for compact action sets (see \cite{MP70}). The only difference is that in (\ref{deftrans}) and (\ref{defpayoff}), the sum has to be replaced by an integral.
\\
If $k^*$ is an absorbing state, we have $v_{\pi}(k^*)=v_\lambda(k^*)$, for any $\pi \in \Delta(\m{N}^*)$ and $\lambda \in [0,1]$. Let $k$ be the only non-absorbing state of the game, and $r \in \m{N}$. In the previous proof, inequalities (\ref{indyn1}) and (\ref{indyn2}) yield:
\begin{equation*}
(v_{\pi^r}-v_{\lambda_r})(k) \leq (1-\lambda_r)\mu_r (v_{\pi^{r+1}}-v_{\lambda_r})(k),
\end{equation*}
where $\mu_r$ is the probability that the game is not absorbed, when Player 1 (resp. 2) plays an optimal strategy $x$ (resp. $y$) in  (\ref{dyn11}) (resp. (\ref{dyn22})). 
In what follows, for simplicity we omit the variable $k$.
\\
Let $\displaystyle \Pi_r=\prod_{m=0}^{r-1} (1-\lambda_r)\mu_r$. Relying on the same steps as in the previous proof, we get the analogous of (\ref{in2}), where $N_0$ is defined similarly:
\begin{equation} \label{majabs}
(v_{\pi}-v_{\lambda_0}) \leq \left(\epsilon+I_{\infty}(\pi) \right)+\sum_{r=1}^{N_0} \Pi_r (v_{\lambda_{r}}-v_{\lambda_{r-1}}).
\end{equation}
Let us majorize the right-hand side. The sequence $(\Pi_r)_{r \geq 0}$ is decreasing. Moreover, for all $r \in \left\{1,2,...,N_0 \right\}$, we have
\begin{equation*} 
\sum_{m \geq r+1} \pi_m \geq \Pi_r \geq \epsilon.
\end{equation*}
 Hence, $\lambda_r  \in \left[0,I_{\infty}(\pi)/\epsilon \right]$. Let $V:=\sup_{\lambda \in  \left[0,I_{\infty}(\pi)/\epsilon \right]} v_\lambda$. We have:
\begin{eqnarray*}
\sum_{r=1}^{N_0} \Pi_r (v_{\lambda_{r}}-v_{\lambda_{r-1}})
&=&\sum_{r=1}^{N_0} \Pi_r v_{\lambda_{r}}-\sum_{r=0}^{N_0-1} \Pi_{r+1} v_{\lambda_{r}}
\\
&=&\sum_{r=1}^{N_0-1} (\Pi_r-\Pi_{r+1})v_{\lambda_r}+\Pi_{N_0} v_{\lambda_{N_0}}
-\Pi_1 v_{\lambda_0} \\
&\leq & V
\sum_{r=1}^{N_0-1} (\Pi_r-\Pi_{r+1})+\Pi_{N_0} v_{\lambda_{N_0}}
-\Pi_1 v_{\lambda_0} \\
&\leq & 
V(\Pi_1-\Pi_{N_0})+\Pi_{N_0} v_{\lambda_{N_0}}
-\Pi_1 v_{\lambda_0} \\
&= & \Pi_1\left(V-v_{\lambda_0}\right)-\Pi_{N_0}\left(V-v_{\lambda_{N_0}}\right)
\\
& \leq & \left|V-v_{\lambda_0}\right|+\left|V-v_{\lambda_{N_0}}\right|.
\end{eqnarray*}
Because $(v_\lambda)$ converges to $v_\infty$ when $\lambda$ goes to $0$, the term $\left|V-v_{\lambda}\right|$ vanishes when $\lambda$ goes to 0. Consequently, the right-hand side of the last inequality goes to $0$ when $I_\infty(\pi)$ goes to 0. Together with (\ref{majabs}), this shows that the positive part of $(v_{\pi}-v_{\lambda_0})$ goes to $0$ when
$I_\infty(\pi)$ goes to 0.
\\
Symmetrically, one can show that the negative part of $(v_{\pi}-v_{\lambda_0})$ goes to $0$ when $I_\infty(\pi)$ goes to 0. Hence, the proposition is proved.
$\squareforqed$
\end{proof}
\subsection{An example} \label{ce}
We construct a stochastic game of order $1/2$, which has no $p$-asymptotic value for any $p>1/2$. First, this shows that our main result (see Theorem \ref{ex}) cannot be improved, second this implies that Theorem \ref{RV} does not extend to the Two-Player Case.

Let us consider the following stochastic game $\Gamma$:
\begin{table}[H]
\label{game}
\centering
\caption{Transition and payoff functions in state $\omega_1$ and $\omega_2$}
   \vspace{0.4cm}
   \setlength{\extrarowheight}{4 pt}
\begin{tabular}{ |l|*{3}{c|}}
\hline 
$\omega_1$
& L & R \\\hline
T & $1$ & $\overrightarrow{1}$ \\\hline
M & 0 & 0 \\\hline
B & $\overrightarrow{1}$ & $1^*$ \\\hline
\end{tabular}
\hspace{1cm}
\begin{tabular}{|l|*{3}{c|}}
\hline 
$\omega_2$
&  L & R \\\hline
T & $0$ & $\overleftarrow{0}$ \\\hline
M & 0 & 0 \\\hline
B & $\overleftarrow{0}$ & $0^*$ \\\hline
\end{tabular}
\end{table}
The set of states of the game is $K=\left\{\omega_1,\omega_2,1^*,0^* \right\}$. The action set is $I=\left\{T,M,B\right\}$ for Player 1 and $J=\left\{L,R\right\}$ for Player 2. States $1^*$ and $0^*$ are absorbing states with absorbing payoff respectively $1$ and $0$. The payoff and transition functions in state $\omega_1$ (resp. $\omega_2$) are described by the left table (resp. the right one). The symbol $\overrightarrow{1}$ (resp. $\overleftarrow{0}$) means that the payoff is $1$ (resp. $0$) and the game moves on to state $\omega_2$ (resp. $\omega_1$). When there is no arrow or star, this means that the game remains in the same state.
\\
In \cite{vigeral13}, a similar stochastic game $\Gamma'$ is mentioned. The only difference is that in $\Gamma'$, Player 1 has only the two actions $T$ and $B$. The uniform value $v_{\infty}'$ of $\Gamma'$ satisfies $\displaystyle v_{\infty}'(\omega_1)=v_{\infty}'(\omega_2)=1/2$. In addition, the order of $\Gamma'$ is $1/2$. Moreover, for all $\epsilon>0$, the stationary strategy $x$ (resp. $y$) for Player 1 (resp. 2) defined by $x(\omega_1)=x(\omega_2)= (1-\sqrt{\lambda}) \cdot T+\sqrt{\lambda} \cdot B$ (resp. $y(\omega_1)=y(\omega_2)=(1-\sqrt{\lambda}) \cdot L+\sqrt{\lambda} \cdot R$) is $\epsilon$-optimal in $\Gamma'_{\lambda}$, for $\lambda$ small enough (they are \textit{asymptotically optimal} strategies). 
\\
In our example, in $\Gamma_{\lambda}$, the action $M$ is dominated by $T$ in every state. Thus for all $\lambda \in (0,1]$, $v_{\lambda}=v'_{\lambda}$. In particular, $\Gamma$ has order $1/2$, and its uniform value $v_{\infty}$ satisfies $\displaystyle v_{\infty}(\omega_1)=v_{\infty}(\omega_2)=1/2$. In addition, the strategy $x$  (resp. $y$) is an asymptotically optimal stationary strategy for Player 1 (resp. 2) in $\Gamma_{\lambda}$. 

\begin{remark} \label{tradeoff} \mbox{}
\begin{itemize} 
\item
Assume that for some $\alpha>0$, Player 1 plays $(1-\alpha) \cdot T+\alpha \cdot B$ in state $\omega_2$ until the state changes. Whatever Player 2 plays, Player 1 spends at most a number of stages of order $\alpha^{-1}$ in state $\omega_2$ before moving to state $\omega_1$ or $0^*$, and the probability that  the state goes to $0^*$ and not to $\omega_1$ is at most of order $\alpha$. Hence, for Player 1 there is a trade-off between staying not too long in state $\omega_2$, and having a low probability of being absorbed in $0^*$. In view of what precedes, the optimal trade-off in $\Gamma_{\lambda}$ is $\alpha \approx \sqrt{\lambda}$. 
\item
Let $\theta \in \Delta(\m{N}^*)$. At some stage $m$ in $\Gamma_{\theta}$, the action $M$ may not be dominated by $T$ in state $\omega_1$. Indeed, if $\theta_m=0$, the stage payoff is 0 whatever the actions played. Thus, it is optimal for Player 1 to play $M$, because it makes the state remain in $\omega_1$. The example builds on this fact. By contrast, for any $\theta \in \Delta(\m{N}^*)$, in $\Gamma_{\theta}$, the action $M$ is dominated by $T$ in state $\omega_2$. In what follows, we build a family of strategies for Player 1 that all use action $M$ in state $\omega_2$, but this is only to make the proof easier. 
\end{itemize}

\end{remark}
\begin{theorem} \label{cep}
For all $p>1/2$, the game $\Gamma$ has no $p$-asymptotic value.
\end{theorem}
The remainder of the subsection is dedicated to the proof of Theorem \ref{cep}. Let us introduce the following piece of notation: given three sequences of strictly positive real numbers $(u_n)_{n \geq 1}$, $(v_n)_{n \geq 1}$ and $(w_n)_{n \geq 1}$, we write $u_n = v_n+o(w_n)$ if the sequence $([u_n-v_n]/{w_n})_{n \geq 1}$ converges to 0.

To simplify the presentation, we first show that $\Gamma$ has no $1$-asymptotic value. Let $n \in \m{N}^*$. For $\l \in \left\{0,1,...,n^3-1\right\}$, define $a_n(l):=l(n+n^5)+1$ and $b_n(l):=l(n+n^5)+n$. Let $\displaystyle E_1:=\cupp_{0 \leq l \leq n^3-1} \left\{a_n(l),a_n(l)+1,...,b_n(l) \right\}$. 
\\
We consider the sequence $\pi^n \in \Delta(\m{N}^*)$ defined by 
\\
$\displaystyle \pi^n_m:=n^{-4}$ if $m \in E_1$, and $\pi^n_m:=0$ otherwise.
We have
\begin{equation*}
I_1(\pi^n)=(2n^3-1)n^{-4},
\end{equation*}
thus $\displaystyle \lim_{n \rightarrow +\infty} I_1(\pi^n)=0$. We show below that $\displaystyle \lim_{n \rightarrow +\infty} v_{\pi^n}(\omega_1)=1$. 
\\
We consider the Markovian strategy $\sigma^n \in \Sigma$ for Player 1, described by the following table:
\begin{table}[H]
\caption{\label{strat} Strategy $\sigma^n$}
\centering
\vspace{0.4cm}
 \setlength{\extrarowheight}{2 pt}
\begin{tabular}{|l|*{3}{c|}}
\hline \backslashbox{state}{stage}
&$m \in E_1$ & $m \notin E_1$ \\\hline
$k_m=\omega_1$ & $\displaystyle \left(1-n^{-2} \right) \cdot T+n^{-2} \cdot B$ & $M$ \\\hline
$k_m=\omega_2$ & $M$ & $\displaystyle \left(1-n^{-4} \right) \cdot T+n^{-4} \cdot B$ \\\hline
\end{tabular}
\end{table}
We show that for any $\epsilon>0$, for any $n$ sufficiently large, $\sigma^n$ guarantees the payoff $1-\epsilon$ in $\Gamma^{\omega_1}_{\pi^n}$ for Player 1.

 Let $\tau^n$ be a pure Markovian best-response to $\sigma^n$ in $\Gamma^{\omega_1}_{\pi^n}$. Let $\Omega_n$ be the event 
\begin{equation*}
\Omega_n:= \capp_{l \in \left\{0,1,...,n^3-1 \right\}} \left\{k_{a_n(l)} \in \left\{\omega_1,1^* \right\} \right\}. 
\end{equation*}

When the state of the game is $\omega_2$ and $m \notin E_1$, Player 1 plays $B$ with probability $n^{-4}$. By Remark \ref{tradeoff}, Player 1 spends at most a number of stages of order $n^4$ in $\omega_2$, and the state goes to $0^*$ with a probability at most of order $n^{-4}$. As a result,  if for some $l \in \left\{0,...,n^3-1 \right\}$ the state is in $\omega_2$ at stage $b_n(l)+1$, the probability that it will move to $\omega_1$ before stage $a_n(l+1)$ is at least of order 
$1-n^{-4}$.
Once the state has moved to $\omega_1$, Player 1 plays $M$ and the state remains in $\omega_1$ until stage $a_n(l+1)$. Hence the probability that $k_{a_n(l)}$ lies in $\left\{\omega_1,1^* \right\}$ for any $l$ in $\left\{0,1,...,n^3-1 \right\}$ is at least of order $(1-n^{-4})^{n^3}= 1+o(1)$. This informal discussion provides intuition for the following lemma. 
\begin{lemma} \label{lemma11}
\begin{equation*}
\lim_{n \rightarrow+\infty} \m{P}^{\omega_1}_{\sigma^n,\tau^n}(\Omega_n)=1.
\end{equation*}
\end{lemma}
For notational convenience, in the proof of this lemma and the proof of the next proposition, for $n \in \m{N}^*$, $\m{P}^{\omega_1}_{\sigma^n,\tau^n}$ is denoted by $\mathbb{P}$ and $\m{E}^{\omega_1}_{\sigma^n,\tau^n}$ is denoted by $\m{E}$.
\begin{proof}

Let $n \in \m{N}^*$ and $l \in \left\{0,...,n^3-1 \right\}$. Let us minorize the probability $\mathbb{P}\left(k_{a_n(l+1)} \in \left\{ \omega_1,1^* \right\}| k_{a_n(l)}  \in \left\{ \omega_1,1^* \right\}\right)$. 
\\
First, notice that $\mathbb{P}\left(k_{b_n(l)+1} \neq 0^* |k_{a_n(l)} \in \left\{ \omega_1,\omega_1^* \right\}\right)=1$ (see Table \ref{strat}). Let us now analyze how the state may evolve during the block 
\\
$\left\{b_n(l)+1,b_n(l)+2,...,a_n(l+1)-1 \right\}$, discriminating between the case 
\\
$k_{b_n(l)}=1^*$, $k_{b_n(l)}=\omega_1$, and $k_{b_n(l)}=\omega_2$:
\begin{itemize}
\item [-]
$\mathbb{P}\left(k_{a_n(l+1)}=1^*|k_{b_n(l)+1}=1^*\right)=1$.
\item[-]
If $k_{b_n(l)+1}=\omega_1$, then Player 1 will play $M$ at each stage 
\\
$m \in \left\{b_n(l)+1,b_n(l)+2,...,a_n(l+1)-1 \right\}$. Therefore, the state will remain in $\omega_1$:
\\
$\mathbb{P}(k_{a_n(l+1)}=\omega_1|k_{b_n(l)+1}=\omega_1)=1$. 
\item[-]
If $k_{b_n(l)+1}=\omega_2$, then Player 1 will play $(1-n^{-4}) \cdot T + n^{-4} \cdot B$ as long as the state is $\omega_2$ and $m \leq a_n(l+1)-1$. We discriminate between two cases:
\begin{itemize}
\item
 if Player 2 plays $L$ as long as the state is $\omega_2$ and $m \leq a_n(l+1)-1$, the game will never be absorbed in $0^*$, and the probability that the state will move to $\omega_1$ before stage $a_n(l+1)$ is equal to 
$\displaystyle 1-\left(1-n^{-4}\right)^{n^5}$. If the state moves to $\omega_1$ at some stage $m \leq a_n(l+1)-1$, then Player 1 will play $M$ until stage $a_n(l+1)$, thus the state will remain in $\omega_1$. Consequently, in this case we have 
\begin{equation*}
\mathbb{P}(k_{a_n(l+1)}=\omega_1|k_{b_n(l)+1}=\omega_2)=1-\left(1-n^{-4}\right)^{n^5}.
\end{equation*}
\item
if Player 2 plays $R$ at one stage in $\left\{b_n(l)+1,b_n(l)+2,...,a_n(l+1)-1 \right\}$, and if at the first stage he does so the state is $\omega_2$, then with probability $1-n^{-4}$ the state will move to $\omega_1$. It will remain in $\omega_1$ until stage $a_n(l+1)$. If the state has already switched to $\omega_1$ before Player $2$ plays $R$, then it will  remain in $\omega_1$ until stage $a_n(l+1)$. Therefore, in this case we have 
\begin{equation*}
\mathbb{P}(k_{a_n(l+1)}=\omega_1|k_{b_n(l)+1}=\omega_2) \geq 1-n^{-4}.
\end{equation*}
\end{itemize}
The last two subcases show that 
\begin{equation*}
\mathbb{P}(k_{a_n(l+1)}=\omega_1|k_{b_n(l)+1}=\omega_2) \geq 
\displaystyle \min \left\{
1-\left(1-n^{-4} \right)^{n^5}, 1-n^{-4} \right\}.
\end{equation*}
\end{itemize}
This exhaustive study shows that
\begin{equation*}
\mathbb{P}(k_{a_n(l+1)} \in \left\{\omega_1,1^* \right\}|k_{a_n(l)} \in \left\{\omega_1,1^* \right\}) \geq 
\displaystyle \min \left\{
1-\left(1-n^{-4}\right)^{n^5}, 1-n^{-4} \right\}.
\end{equation*}
We have $\left(1-n^{-4}\right)^{n^5}=o(n^{-4})$, thus for $n$ large enough, the minimum in the above equation is reached at $1-n^{-4}$. 
 By induction, it yields
\begin{equation*}
\m{P}(\Omega_n) \geq \prod_{l=0}^{n^3-1} (1-n^{-4}) = 1+o(1),
\end{equation*}
and the lemma is proved.
$\squareforqed$
\end{proof}
Now we can prove the following proposition:
\begin{proposition}
The game $\Gamma$ has no $1$-asymptotic value.
\end{proposition}
\begin{proof}
Let $n \in \m{N}^*$. We minorize $\gamma_{\pi^n}^{\omega_1}(\sigma^n,\tau^n)$ by a quantity that goes to $1$ as $n$ goes to infinity. 
\\
The last lemma shows that with high probability, at the beginning of each block $\left\{a_n(l),a_n(l)+1,...,b_n(l) \right\}$, the state is either $\omega_1$ or $1^*$. Recall that these blocks exactly correspond to the stages where the payoff weight is nonzero. Hence, to get a good payoff between stage $a_n(l)$ and stage $b_n(l)$, Player 2 should make the state move from $\omega_1$ to $\omega_2$ at least before stage $b_n(l)$. If Player 2 plays $L$ at each stage $m \in \left\{a_n(l),a_n(l)+1,...,b_n(l) \right\}$, with probability greater than $1-(1-n^{-2})^n$, the state will remain in $\omega_1$ until stage $b_n(l)$. This probability goes to $1$ as $n$ goes to infinity, which is a bad outcome for Player 2. Thus, Player 2 should play $R$ at some stage $m \in \left\{a_n(l),a_n(l)+1,...,b_n(l) \right\}$. We show that:
\begin{itemize}
\item[-]
either the number of $l \in \left\{0,1,...,n^3-1 \right\}$ such that Player 2 plays at least one time $R$ 
in $\left\{a_n(l),a_n(l)+1,...,b_n(l) \right\}$ is small, and thus the total payoff in $\Gamma^{\omega_1}_{\pi^n}$ is close to $1$,
\item[-]
either the number of $l \in \left\{0,1,...,n^3-1 \right\}$ such that Player 2 
plays at least one time $R$ in $\left\{a_n(l),a_n(l)+1,...,b_n(l) \right\}$ is high. In this case, with probability close to 1, the state is absorbed in $1^*$ very rapidly, thus the total payoff in $\Gamma^{\omega_1}_{\pi^n}$ is close to 1. 
\end{itemize}
Let $n \in \m{N}^*$ and $l \in \left\{0,1,...,n^3-1 \right\}$, and $\Omega_n(l)$ be the event defined by
\begin{equation*}
\Omega_n(l):=\capp_{0 \leq l' \leq l} \left\{ k_{a_n(l')} \in \left\{\omega_1,1^* \right\} \right\}.
\end{equation*}
Note that $\Omega_n(n^3-1)=\Omega_n$. Let 
\begin{equation*}
M_n(l):=\left\{l' \in \left\{0,1,...,l \right\} \ | \ \exists m \in \left\{a_n(l),a_n(l)+1,...,b_n(l) \right\}, \tau^n(m,\omega_1)=R \right\},
\end{equation*}
and let $\overline{M_n(l)}:=\left\{0,1,...,l \right\} \setminus M_n(l)$.
If $l \in M_{n}(n^3-1)$, let 
\begin{equation*}
m_n(l):=\min \left\{ m \in \left\{a_n(l),a_n(l)+1,...,b_n(l) \right\} \ | \ \tau^n(m,\omega_1)=R \right\}. 
\end{equation*}
Fix $\delta \in (0,1]$. 
Let $l_n:=\max \left\{l \in \left\{0,1,...,n^3-1 \right\} \ | \ \card M_n(l) \leq \delta n^3 \right\}$. We show that between stages $1$ and $b_n(l_n)$, Player 2 did not play $R$ a sufficient number of times to impact the total payoff, and at stage $b_n(l_n)+1$, either $l_n=n^3-1$ and the game is finished, or he has played too many times $R$, in such a way that the state has been absorbed in $1^*$ with high probability.
\\
By definition of $l_n$, we have $\card M_n(l_n) \leq \delta n^3$, and if $l_n<n^3-1$, then $\card M_n(l_n) \geq \delta n^3-1$. 
\\
We have
\begin{eqnarray} \nonumber
\m{E}\left(\sum_{m=1}^{b_n(l_n)} \pi^n_m g_m \right) &=& 
\frac{1}{n^4}\sum_{l=0}^{l_n} \m{E}\left(
\sum_{m=a_n(l)}^{b_n(l)} g_m \right)
\\
\label{inpay1}
&\geq &
\frac{1}{n^4} \sum_{l \in \overline{M_n(l_n)}} \m{E}\left(
1_{\Omega_n(l)}
\sum_{m=a_n(l)}^{b_n(l)} g_m \right).
\end{eqnarray}
If $l \in \overline{M_n(l_n)}$ and $k_{a_n(l)}=\omega_1$, Player 2 plays $L$ as long as $k_m=\omega_1$ and 
\\
$m \leq b_n(l)$, while Player 1 plays $(1-n^{-2}) \cdot T + n^{-2} \cdot B$. As a result, if $k_{a_n(l)}=\omega_1$, the probability that the state remains in $\omega_1$ until stage $b_n(l)$ is 
\\
$\alpha_n:=(1-n^{-2})^n$. Thus, the last inequality yields
\begin{eqnarray} \label{payoff11}
\m{E}\left(\sum_{m=1}^{b_n(l_n)} \pi^n_m g_m \right) &\geq&
n^{-3} \card \overline{M_n(l_n)} \m{P}(\Omega_n) \alpha_n
\\
&\geq& (n^{-3}(l_n+1)-\delta) \m{P}(\Omega_n) \alpha_n \label{payoff1}.
\end{eqnarray}
\begin{case} $l_n=n^3-1$. \end{case}
By (\ref{payoff11}) and Lemma \ref{lemma11}, there exists $\bar{n} \in \m{N}^*$ such that for all $n \geq \bar{n}$ verifying $l_n=n^3-1$, 
\begin{equation} \label{case1}
\gamma^{\omega_1}_{\pi^n}(\sigma^n,\tau^n) \geq 1-2\delta .
\end{equation}
\begin{case} $l_n<n^3-1$. \end{case}
Let $n \in \m{N}^*$ such that $l_n<n^3-1$. In particular, $\left|M_n(l_n) \right| \geq \delta n^3-1$ and $\left|\overline{M_n(l_n)} \right| \leq l_n-\delta n^3+2$.

We are going to show the following inequality: 
\begin{equation} \label{lemma}
\m{P}(k_{b_n(l_n)}=1^*) \geq \m{P}(\Omega_n)-\left(1-n^{-2}\left(1-n^{-2}\right)^n\right)^{\delta n^3 -1}:=\beta_n .
\end{equation}
The idea is the following.  Each time Player 2 plays $R$ in state $\omega_1$, the state goes to $1^*$ with probability $n^{-2}$. If $k_{a_n(l)}=\omega_1$ and $l \in M_n(l)$, then at each stage $m \in \left\{a_n(l),a_n(l)+1,...,m_n(l)-1 \right\}$, Player 2 will play $L$, hence at each of these stages the state will remain in $\omega_1$  with probability $n^{-2}$. Since $m_n(l)-a_n(l) \leq n$, with high probability $k_{m_n(l)}=\omega_1$. At stage $m_n(l)$, Player 2 plays $R$. Thus with high probability, conditionnal to the event $\Omega_n(l_n)$, before stage $b_n(l_n)$ Player 2 has played more than $\delta n^3-1$ times the action $R$ in state $\omega_1$, leading the state to be absorbed in $1^*$ before stage $b_n(l_n)$ with high probability. 

Formally, if $l \in \left\{0,1,...,l_n \right\}$, we have
\begin{eqnarray} 
\nonumber
\m{P}(\left\{k_{b_n(l)} \neq 1^* \right\} \cap \Omega_n(l))&=&
\m{P}(\left\{k_{b_n(l)} \neq 1^* \right\}  \cap \left\{k_{a_n(l)} = \omega_1 \right\} \cap \Omega_n(l))
\\
\nonumber
&=&\m{P}(\left\{k_{b_n(l)} \neq 1^*\right\} |\left\{k_{a_n(l)} = \omega_1 \right\}\cap \Omega_n(l))
\\
\label{prod}
&\times& \m{P}(\left\{k_{a_n(l)} = \omega_1 \right\} \cap \Omega_n(l)).
\end{eqnarray}
First we majorize the first term $P_1:=\m{P}(k_{b_n(l)} \neq 1^*|\left\{k_{a_n(l)} = \omega_1 \right\} \cap \Omega_n(l))$. If  $l \notin M_n(l)$, we simply majorize it by 1. Assume now that $l \in M_n(l)$. 
By the Markov property ($\sigma^n$ and $\tau^n$ are Markovian strategies), we have
\begin{eqnarray*}
P_1
&=& \m{P}(k_{b_n(l)} \neq 1^*|\left\{k_{a_n(l)}=\omega_1 \right\})
\\
&=&\m{P}(\left\{k_{b_n(l)} \neq 1^*\right\} \cap \left\{k_{m_n(l)}=\omega_1\right\}|\left\{k_{a_n(l)}=\omega_1 \right\})
\\
&+& \m{P}(\left\{k_{b_n(l)} \neq 1^*\right\} \cap \left\{k_{m_n(l)} \neq \omega_1\right\}|\left\{k_{a_n(l)}=\omega_1 \right\}).
\end{eqnarray*}
Let $P_3:=\m{P}(k_{m_n(l)} \neq \omega_1 | k_{a_n(l)}=\omega_1)$. The last equality and the Markov property give
\begin{eqnarray} 
\nonumber
P_1 &\leq& \m{P}(k_{b_n(l)} \neq 1^*  | \left\{k_{m_n(l)}=\omega_1 \right\} \cap \left\{k_{a_n(l)}=\omega_1 \right\}) (1-P_3)+P_3
\\
\label{inprob1}
&=& \m{P}(k_{b_n(l)} \neq 1^*  | k_{m_n(l)}=\omega_1) (1-P_3)+P_3.
\end{eqnarray}
If $k_{m_n(l)}=\omega_1$, then at stage $m_n(l)$ Player 2 plays the action $R$, hence the state is absorbed in $1^*$ with probability $\displaystyle n^{-2}$. Thus 
\begin{equation} \label{inprob2}
\m{P}(k_{b_n(l)} \neq 1^*  | k_{m_n(l)}=\omega_1) \leq 1-n^{-2}.
\end{equation}
If $k_{a_n(l)}=\omega_1$, then at each stage $m \in \left\{a_n(l),a_n(l)+1,...,m_n(l)-1 \right\}$, Player 2 will play $L$, hence at each stage the state will remain in $\omega_1$  with probability $1-n^{-2}$, and $m_n(l)-a_n(l) \leq n$. We deduce that
\begin{equation} \label{inprob3}
P_3 \leq 1-\left(1-n^{-2} \right)^n.
\end{equation}
Combining (\ref{inprob1}), (\ref{inprob2}) and (\ref{inprob3}) gives
\begin{eqnarray}
\nonumber
P_1 &\leq& \left(1-n^{-2}\right)(1-P_3)+P_3
\\ \nonumber
&=&1+ n^{-2}(P_3-1)
\\
\label{P1m}
&\leq& 1-n^{-2}\left(1-n^{-2}\right)^n.
\end{eqnarray}
As for the second term in (\ref{prod}), we have
\begin{equation}
\label{P2m}
\m{P}(\left\{k_{a_n(l)} = \omega_1 \right\} \cap \Omega_n(l)) 
\leq \m{P}(\left\{k_{b_n(l-1)} \neq 1^* \right\} \cap \Omega_n(l-1)).
\end{equation}
Combining (\ref{prod}), (\ref{P1m}) and (\ref{P2m}), we deduce that if $l \in M_n(l)$, then
\begin{equation*}
\m{P}(\left\{k_{b_n(l)} \neq 1^* \right\} \cap \Omega_n(l)) \leq \left(1-n^{-2}\left(1-n^{-2}\right)^n \right) \m{P}(\left\{k_{b_n(l-1)} \neq 1^* \right\} \cap \Omega_n(l-1)).
\end{equation*}
Because $\left|M_n(l_n) \right| \geq \delta n^3-1$, by induction we obtain
\begin{equation*}
\m{P}(\left\{k_{b_n(l_n)} \neq 1^* \right\} \cap \Omega_n(l_n)) \leq \left(1-n^{-2}\left(1-n^{-2}\right)^n\right)^{\delta n^3 -1},
\end{equation*}
and inequality (\ref{lemma}) follows.
Now we can minorize the other part of the payoff:
\begin{eqnarray}
\nonumber
\m{E}\left(\sum_{m \geq b_n(l_n)+1} \pi^n_m g_m \right) &\geq &
\m{E}\left(1_{\left\{k_{b_n(l_n)}=1^* \right\}}\sum_{m \geq b_n(l_n)+1}\pi^n_m \right)
\\ 
\nonumber
&=&n^{-3}(n^3-l_n-1) \m{P}(\left\{k_{b_n(l_n)}=1^*\right\})
\\
\label{payoff2}
&\geq & \left(1-n^{-3}(l_n+1) \right) \beta_n.
\end{eqnarray}
Inequalities (\ref{payoff1}) and (\ref{payoff2}) yield
\begin{eqnarray*}
\m{E}\left(\sum_{m \geq 1} \pi^n_m g_m \right) &=&
\m{E}\left(\sum_{m=1}^{b_n(l_n)} \pi^n_m g_m \right)+\m{E}\left(\sum_{m \geq b_n(l_n)+1} \pi^n_m g_m \right)
\\
&\geq & (n^{-3}(l_n+1)-\delta) \m{P}(\Omega_n) \alpha_n+ \left(1-n^{-3}(l_n+1) \right) \beta_n.
\end{eqnarray*}
\\
The sequences $(\alpha_n)_{n \geq 1}$, $(\beta_n)_{n \geq 1}$ and $(\m{P}(\Omega_n))_{n \geq 1}$ converge to $1$, thus there exists $n_1 \in \m{N}^*$ such that for all $n \geq n_1$ verifying $l_n<n^3-1$, we have
\begin{equation} \label{case2}
v_{\pi^n}(\omega_1) \geq \gamma^{\omega_1}_{\pi^n}(\sigma^n,\tau^n) \geq 1-2\delta.
\end{equation}

Because $\tau^n$ is a best-response strategy to $\sigma^n$ in $\Gamma^{\omega_1}_{\pi_n}$,  inequalities (\ref{case1}) and (\ref{case2}) show that for $n \geq \max(\bar{n},n_1)$, we have
\begin{equation} \label{minpayoff}
v_{\pi^n}(\omega_1) \geq \gamma^{\omega_1}_{\pi^n}(\sigma^n,\tau^n) \geq 1-2\delta.
\end{equation}
Because $\delta \in (0,1]$ is arbitrary, 
the sequence $(v_{\pi^n}(\omega_1))_{n \geq 1}$ converges to $1$, and $\Gamma$ has no $1$-asymptotic value. 
$\squareforqed$
\end{proof}
Now we can prove Theorem \ref{cep}.
\begin{proof}[Proof of Theorem \ref{cep}]
Let $\epsilon>0$ and $p:=1/2+\epsilon$. Proving that $\Gamma$ has no $p$-asymptotic value proceeds in the same way as previously. The only difference is that the sequence of weights $(\pi^n)$ has to be modified.
Let $\epsilon>0$ and $n \in \m{N}^*$. In what follows, the integer part of a real number $x$ is denoted by $\left\lfloor x \right\rfloor$. 
Define two integers $N_1$ and $N_2$ by
\begin{equation*}
N_1:=\lfloor n^{2-\epsilon} \rfloor \quad \text{and} \quad N_2:=\lfloor n^{2+\epsilon} \rfloor.
\end{equation*}
For $l \in \left\{0,1,...,N_2-1 \right\}$, let $a_n'(l):=l(N_1+n^5)+1$ and $b_n'(l):=l(N_1+n^5)+N_1$. Let 
\begin{equation*}
E'_1:=\cupp_{l \in \left\{0,1,...,N_2 \right\}} \left\{a_n'(l),a_n'(l)+1,...,b_n'(l) \right\}.
\end{equation*}
Let $\pi'^n \in \Delta(\m{N}^*)$ defined in the following way: for $m \in \m{N}^*$,

\begin{equation*}
\pi'^n_m:= \left\{
\begin{array}{ll}
n^{-4}  & \mbox{if} \ \ m \in E'_1 \setminus \left\{N_1+1 \right\}, \\
1-\displaystyle \sum_{m \neq N_1+1} \pi'^n_m & \mbox{if} \ \ m=N_1+1, \\
0 & \mbox{if} \ \ m \notin E_1'.
\end{array}
\right.
\end{equation*}
We have
\begin{equation*}
I_{p}(\pi'^n) \leq \lfloor n^{2+\epsilon} \rfloor n^{-4 (1/2+\epsilon)}+2 \pi'^n_{N_1+1}.
\end{equation*}
Hence $\displaystyle \lim_{n \rightarrow+\infty} I_p(\pi'^n)=0$. We claim that $\displaystyle \lim_{n \rightarrow +\infty} v_{\pi'^n}(\omega_1)=1$. The proof is the same as above. We still consider the same strategy $\sigma^n$ for Player 1 in $\Gamma^{\pi'^n}$.  Lemma \ref{lemma11} is still true. 
Indeed, the length of the blocks 
\\
$\left\{b'_n(l)+1,b'_n(l)+2,...,a'_n(l+1)-1 \right\}$ is still $n^5$.
\\
Now let us check the remainder of the proof. The quantity $(1-n^{-2})^{n^{2-\epsilon}}$ goes to 1 as $n$ goes to infinity. Hence if $l \in \left\{0,...,N_2 \right\}$ and $k_{l(N_1+n^5)+1}=\omega_1$, 
to get a good payoff between stage $a'_n(l)$ and stage $b'_n(l)$, Player 2 should make the state move from $\omega_1$ to $\omega_2$ at least before stage $b_n(l)$. Thus, he has to play $R$ at least one time, and take a risk of being absorbed in $\omega_1$ of $n^{-2}$. There are approximately $n^{2+\epsilon}$ such blocks. Since $(1-n^{-2})^{n^{2+\epsilon}}$ goes to $0$ as $n$ goes to infinity, the same proof as before shows that, when $n$ goes to infinity, 
the sequence $(v_{\pi^n}(\omega_1))_{n \geq 1}$ converges to $1$.
$\squareforqed$
\end{proof}

\section{Uniform approach}
To relax the assumption that sequences of weights are decreasing in Corollary \ref{MNd}, the simplest sequences of weights one can imagine are the $\pi^{l,n}$ defined in Remark \ref{rq}: $\displaystyle \pi^{l,n}:=n^{-1} 1_{l+1 \leq m \leq l+n}$. As we have seen, Theorem \ref{ex} shows that for any stochastic game, 
\begin{equation*}
\lim_{n \rightarrow +\infty} \sup_{l \in \m{N}} \left\|v_{\pi^{l,n}}-v_{\infty} \right\|_{\infty}=0.
\end{equation*}
Is it possible to show the existence of strategies that are approximately optimal in any game $\Gamma_{\pi^{l,n}}$, for any $l \geq 0$ and $n$ big enough, for both players? We provide an example of an absorbing game where this property does not hold. Thus, no natural extension of Theorem \ref{MNd} to sequences of weights which are not decreasing seems to exists.
\\

Consider the following absorbing game, introduced by \cite{gillette57} under the name of "Big Match".
The state space is $K=\left\{\omega,1^*,0^* \right\}$, where $1^*$ (resp. $0^*$) is an absorbing state with payoff 1 (resp. 0).
Action sets for Player 1 and 2 are respectively $I=\left\{T,B \right\}$ and $J=\left\{L,R \right\}$.
The payoff and transition functions in state $\omega$ are described by the following table:
\begin{figure}[H]
\centering
\caption{Transition and payoff functions in state $\omega$}
   \vspace{0.4cm}
\begin{tabular}{|l|*{3}{c|}}
\hline \backslashbox{J1}{J2}
& $L$ & $R$ \\\hline
$T$ & $1^*$ & $0^*$ \\\hline
$B$ & $0$ & $1$ \\\hline
\end{tabular}
\end{figure}

As any stochastic game, the Big Match has a uniform value $v_{\infty}$, and $v_{\infty}(\omega)=1/2$ (see \cite[Chapitre 5, p.93]{S02b}). The stationary strategy $\displaystyle 1/2 \cdot L+1/2 \cdot R$ is a $0$-optimal uniform strategy for Player $2$. Constructing $\epsilon$-optimal uniform strategy for Player 1 is more tricky (see \cite{BF68}).

Now we investigate the general uniform approach in the Big Match. 
\begin{definition}
Let $\Gamma$ be a stochastic game, and $k_1$ the initial state. Player 1 can guarantee uniformly \textit{in the general sense} $\alpha \in \m{R}$ in $\Gamma^{k_1}$ if for all $\epsilon>0$, there exists $\sigma^* \in \Sigma$ and $N_1 \in \mathbb{N}^*$ such that for all $\tau \in \mathcal{T}$, $n \geq N_1$ and $\bar{n} \in \mathbb{N}$, we have
\begin{equation} \label{gar}
\mathbb{E}^{k_1}_{\sigma^*,\tau}\left(\frac{1}{n}\sum_{m=\bar{n}+1}^{\bar{n}+n} g_m \right) \geq \alpha-\epsilon.
\end{equation}
\end{definition}
First, let us explain why Player 1 can not guarantee uniformly more than 0 in the general sense. Assume the contrary: Player 1 can guarantee uniformly in the general sense $\alpha>0$. Let $(N_1,\sigma^*) \in \m{N}^* \times \Sigma$ corresponding to $\displaystyle \epsilon=\alpha/2$ in (\ref{gar}).
The stationary strategy $\alpha/10 \cdot L+(1-\alpha/10) \cdot R$ is denoted by $y$.
\\
On the one hand, Player 1 should not play $T$ at any stage of the game, against the strategy $y$. On the other hand, in the infinitely repeated game, with high probability Player $2$ will play $L$ at $N_1$ random consecutive stages.
At that point, Player 1 does not know if Player $2$ has switched to a pure strategy that plays $L$ at any stage, or if he still plays $y$. In the first case, Player $1$ should play $T$ at least one time during these $N_1$ stages, but in the second case, he should not play $T$. Thus, he cannot guarantee a good payoff against both strategies. This provides intuition for the following proposition:
\begin{proposition}
Player 1 cannot guarantee uniformly more than 0 in the general sense.
\end{proposition}
\begin{remark}
This proposition shows in particular that Theorem \ref{RV} (Renault and Venel) does not generalize to zero-sum stochastic games, even for absorbing games: the game $\Gamma$ has no general uniform value.
\end{remark}
\begin{proof}
The same notations as in the above discussion are used. 

For $n \in \m{N}^*$, let $A_n$ be the event $\left\{\text{Player 1 plays $T$ before stage $n$} \right\}$, and let $\overline{A_n}$ be the complement of $A_n$. The sequence $\left(\m{P}_{\sigma^*,y}(A_n)\right)_{n \geq 1}$ is increasing and bounded by 1, therefore it converges to some $l \in [0,1]$. Let $N_0 \in \m{N}^*$ such that for all $n \geq N_0$, 
\begin{equation*}
\m{P}_{\sigma^*,y}(A_n) \geq l-\alpha/10.
\end{equation*}
To avoid confusion, in what follows $h$ denotes an element of $H_{\infty}$, and $\widetilde{h}$ denotes the random variable with values in $H_{\infty}$ describing the infinite history of the game.
Let $n \geq N_0+N_1$. Let $H^n \subset H_{\infty}$ be the following set:
\begin{equation*}
\left\{h \in H_{\infty}  |  \exists \ a(h) \in \left\{N_0,...,n-N_1 \right\}, \ \forall m \in \left\{a(h)+1,...,a(h)+N_1\right\},  j_m=L \right\}.
\end{equation*}

There exists $N_2 \in \m{N}^*$ such that 
\begin{equation} \label{LGN}
\mathbb{P}_{\sigma^*,y}\left(\widetilde{h} \in H^{N_2}\right) \geq 1/2. 
\end{equation}
We have
\begin{eqnarray}
\label{minprob}
\m{P}_{\sigma^*,y}(\overline{A_{N_0}} \cap A_{N_2}) 
\label{abs}
&\geq & \m{E}_{\sigma^*,y}\left(1_{\left\{\widetilde{h} \in H^n \right\}} \m{P}_{\sigma^*,y'(\widetilde{h})}(\overline{A_{N_0}} \cap A_{N_2}) \right),
\end{eqnarray}
where for $h \in H^n$, the strategy $y'(h)$ is the Markov strategy equal to $y$ between stages $1$ and $a(h)$, and equal to $j_m(h)$ for each stage $m \geq a(h)+1$. 
\\
Let $h \in H^n$.
Let us now minorize $\m{P}_{\sigma^*,y'(h)}(\overline{A_{N_0}} \cap A_{N_2})$. 
\\
If Player 1 plays $T$ at some stage $m \leq a(h)$ against the strategy $y'(h)$, the game is absorbed in $1^*$ with probability $\displaystyle \alpha/10$, and in $0^*$ with probability $\displaystyle 1-\alpha/10$. Therefore we have
\begin{equation} \label{pay1}
\m{E}_{\sigma^*,y'(h)}\left(1_{A_{N_0}} \frac{1}{N_1} \sum_{m=a(h)+1}^{a(h)+N_1} g_m \right) \leq \frac{\alpha}{10}.
\end{equation}
Because $h \in H^n$, we have
\begin{equation} \label{pay2}
\m{E}_{\sigma^*,y'(h)}\left(1_{\overline{A_{N_2}}} \frac{1}{N_1} \sum_{m=a(h)+1}^{a(h)+N_1} g_m \right)=0.
\end{equation}
Combining (\ref{pay1}) and (\ref{pay2}), we obtain
\begin{eqnarray*}
\m{E}_{\sigma^*,y'(h)} \left(\frac{1}{N_1} \sum_{m=a(h)+1}^{a(h)+N_1} g_m \right)
&\leq& \frac{\alpha}{10}+\m{P}_{\sigma^*,y'(h)}(\overline{A_{N_0}} \cap A_{N_2}),
\end{eqnarray*}
and by (\ref{gar}), 
\begin{equation*}
\m{P}_{\sigma^*,y'(h)}(\overline{A_{N_0}} \cap A_{N_2}) \geq \alpha/2 - \alpha/10= 2 \alpha/5.
\end{equation*}
Plugging the last inequality into (\ref{minprob}) and using $(\ref{LGN})$, we deduce that
\begin{eqnarray*}
\m{P}_{\sigma^*,y}(\overline{A_{N_0}} \cap A_{N_2}) &\geq &
\m{E}_{\sigma^*,y}\left(1_{\left\{h \in H^n \right\}}2 \alpha/5 \right)
\\
& \geq & \alpha/5.
\end{eqnarray*}
Because $A_{N_0} \subset A_{N_2}$, we have
\begin{equation*}
\m{P}_{\sigma^*,y}(\overline{A_{N_0}} \cap A_{N_2})=
\m{P}_{\sigma^*,y}(A_{N_2})-\m{P}_{\sigma^*,y}(A_{N_0})\leq \alpha/10,
\end{equation*}
thus $\displaystyle \alpha/10 \geq \alpha/5$, which is a contradiction.
$\squareforqed$
\end{proof}
\section*{Acknowledgments}
I would like to thank J\'{e}r\^ome Renault, Sylvain Sorin, Rida Laraki, and Fabien Gensbittel for their interesting suggestions. 
\newpage
\bibliography{biblio3}

\begin{thebibliography}{14}
\providecommand{\natexlab}[1]{#1}
\providecommand{\url}[1]{\texttt{#1}}
\expandafter\ifx\csname urlstyle\endcsname\relax
  \providecommand{\doi}[1]{doi: #1}\else
  \providecommand{\doi}{doi: \begingroup \urlstyle{rm}\Url}\fi

\bibitem[Bewley and Kohlberg(1976)]{BK76}
T.~Bewley and E.~Kohlberg.
\newblock The asymptotic theory of stochastic games.
\newblock \emph{Mathematics of Operations Research}, 1\penalty0 (3):\penalty0
  197--208, 1976.

\bibitem[Blackwell and Ferguson(1968)]{BF68}
D.~Blackwell and T.~Ferguson.
\newblock The big match.
\newblock \emph{The Annals of Mathematical Statistics}, 39\penalty0
  (1):\penalty0 159--163, 1968.

\bibitem[Cardaliaguet et~al.(2012)Cardaliaguet, Laraki, and Sorin]{CLS12}
P.~Cardaliaguet, R.~Laraki, and S.~Sorin.
\newblock A continuous time approach for the asymptotic value in two-person
  zero-sum repeated games.
\newblock \emph{SIAM Journal on Control and Optimization}, 50\penalty0
  (3):\penalty0 1573--1596, 2012.

\bibitem[Gillette(1957)]{gillette57}
Gillette.
\newblock Stochastic games with zero stop probabilities.
\newblock In \emph{Contributions to the Theory of Games}, pages 179--187.
  Princeton University Press, 1957.

\bibitem[Maitra and Parthasarathy(1970)]{MP70}
A.~Maitra and T.~Parthasarathy.
\newblock On stochastic games.
\newblock \emph{Journal of Optimization Theory and Applications}, 5\penalty0
  (4):\penalty0 289--300, 1970.

\bibitem[Mertens and Neyman(1981)]{MN81}
J.~Mertens and A.~Neyman.
\newblock Stochastic games.
\newblock \emph{International Journal of Game Theory}, 10\penalty0
  (2):\penalty0 53--66, 1981.

\bibitem[Mertens et~al.(2009)Mertens, Neyman, and Rosenberg]{MNR09}
J.-F. Mertens, A.~Neyman, and D.~Rosenberg.
\newblock Absorbing games with compact action spaces.
\newblock \emph{Mathematics of Operations Research}, 34\penalty0 (2):\penalty0
  257--262, 2009.

\bibitem[Neyman and Sorin(2010)]{NS10}
A.~Neyman and S.~Sorin.
\newblock Repeated games with public uncertain duration process.
\newblock \emph{International Journal of Game Theory}, 39\penalty0
  (1):\penalty0 29--52, 2010.

\bibitem[Oliu-Barton(2014)]{B14}
M.~Oliu-Barton.
\newblock The asymptotic value in finite stochastic games.
\newblock \emph{Mathematics of Operations Research}, 39\penalty0 (3):\penalty0
  712--721, 2014.

\bibitem[Renault and Venel(2012)]{RV12}
J.~Renault and X.~Venel.
\newblock A distance for probability spaces, and long-term values in markov
  decision processes and repeated games.
\newblock \emph{Arxiv preprint arXiv:1202.6259}, 2012.

\bibitem[Shapley(1953)]{SH53}
L.~Shapley.
\newblock Stochastic games.
\newblock \emph{Proceedings of the National Academy of Sciences of the United
  States of America}, 39\penalty0 (10):\penalty0 1095, 1953.

\bibitem[Sorin(2002)]{S02b}
S.~Sorin.
\newblock \emph{A first course on zero-sum repeated games}, volume~37.
\newblock Springer, 2002.

\bibitem[Vigeral(2013)]{vigeral13}
G.~Vigeral.
\newblock A zero-sum stochastic game with compact action sets and no asymptotic
  value.
\newblock \emph{Dynamic Games and Applications}, 3\penalty0 (2):\penalty0
  172--186, 2013.

\bibitem[Ziliotto(2013)]{Z13}
B.~Ziliotto.
\newblock Zero-sum repeated games: counterexamples to the existence of the
  asymptotic value and the conjecture maxmin= lim v(n).
\newblock \emph{arXiv preprint arXiv:1305.4778, to appear in Annals of
  Probability}, 2013.

\end{thebibliography}

\end{document}